\documentclass[12pt,a4]{amsart}
\usepackage{latexsym}
\usepackage{amsmath}
\usepackage{amssymb}
\usepackage{amsthm}
\usepackage{amscd}
\usepackage{color}
\definecolor{dblue}{rgb}{0,0,0.45}
\definecolor{red}{rgb}{0.7,0,0}

\usepackage{color}

\numberwithin{equation}{section}

\newtheorem{theorem}{Theorem}[section]
\newtheorem{lemma}[theorem]{Lemma}

\newtheorem{proposition}[theorem]{Proposition}

\theoremstyle{definition}

\newtheorem{remark}[theorem]{Remark}
\newtheorem{definition}[theorem]{Definition}

\theoremstyle{remark}

\begin{document}

\title
[Hardy-type and Heisenberg's Inequality in Morrey Spaces]
{Hardy-type and Heisenberg's Inequality in Morrey Spaces}

\author[H.~Gunawan]{Hendra Gunawan}
\address{Department of Mathematics, Bandung Institute of Technology,
Bandung 40132, Indonesia}
\email{hgunawan@math.itb.ac.id}
\author[D.I.~Hakim]{Denny Ivanal Hakim}
\address{Department of Mathematics and Information Sciences, Tokyo
Metropolitan University, 1-1 Minami Ohsawa, Hachioji, Tokyo 192-0397,
Japan}
\email{dennyivanalhakim@gmail.com}
\author[E.~Nakai]{Eiichi Nakai}
\address{Department of Mathematics, Ibaraki University, Mito, Ibaraki
310-8512}
\email{eiichi.nakai.math@vc.ibaraki.ac.jp}
\author[Y.~Sawano]{Yoshihiro~Sawano}
\address{Department of Mathematics and Information Sciences, Tokyo
Metropolitan University, 1-1 Minami Ohsawa, Hachioji, Tokyo 192-0397,
Japan; and RDUN, Peoples' Friendship University of Russia,
117198, Moscow Miklukho-Maklaya str. 6, Moscow, Russia.}
\email{ysawano@tmu.ac.jp}

\thanks{The first author was supported by ITB Research \&
Innovation Program 2017.
The second and third authors were supported by Grant-in-Aid
for Scientific Research (B) No.~15H03621, Japan Society for
the Promotion of Science. The fourth author is supported by
Grant-in-Aid for Scientific Research (C) No.~16K05209,
Japan Society for the Promotion of Science, and
People's Friendship University of Russia.}

\begin{abstract}
We use the Morrey norm estimate for the imaginary power of the Laplacian
to prove an interpolation inequality for the fractional power of the
Laplacian on Morrey spaces. We then prove a Hardy-type inequality and use
it together with the interpolation inequality to obtain a Heisenberg-type
inequality in Morrey spaces.

\medskip

\end{abstract}


\maketitle

\section{Introduction}

Inspired by the work of Ciatti, Cowling, and Ricci \cite{Ciatti15},
we are interested in obtaining an estimate for the Morrey norm of
the fractional power of the Laplacian, in order to prove Heisenberg's
uncertainty inequality in Morrey spaces. To begin with, let
$(-\Delta)^{z/2}$ be the complex power of the Laplacian, given by
\begin{align}\label{eq:17917-2}
[(-\Delta)^{z/2}f]\,\widehat{\empty}\,(\xi)
:=
|\xi|^z \widehat{f}(\xi),
\quad \xi\in \mathbb{R}^n,
\end{align}
for suitable functions $f$ on $\mathbb{R}^n$, where the Fourier
transform is defined by
\[
\widehat{f}(\xi):=\int_{\mathbb{R}^n} f(x)e^{-ix\cdot\xi} dx,
\quad \xi\in \mathbb{R}^n.
\]
Our first aim here is to show the following Morrey norm estimate
for the imaginary power of the Laplacian:
\begin{equation}\label{est-1}
\|(-\Delta)^{i u/2}f\|_{{\mathcal M}^p_q}
\lesssim
(1+|u|)^{\frac{n}{2}}\|f\|_{{\mathcal M}^p_q},
\quad f \in {\mathcal M}^p_q(\mathbb{R}^n),
\end{equation}
for every $u\in\mathbb{R}$, provided that $1<p\le q<\infty$.

Recall that, for $1\le p \le q<\infty$, the Morrey space
${\mathcal M}^p_q(\mathbb{R}^n)$ is the set of all
$f \in L^p_{\rm loc}({\mathbb R}^n)$ for which
\[
\|f\|_{{\mathcal M}^p_q}
:=
\sup_{a \in {\mathbb R}^n, r>0}
|B(a,r)|^{\frac1q-\frac1p}
\left(\int_{B(a,r)}|f(y)|^p\,{\rm d}y\right)^{\frac1p}
\]
is finite. We refer the reader to \cite{YSY10} for
various function spaces built on Morrey spaces.

Based on \cite{RT14},
let us explain why $(-\Delta)^{iu/2}$ should be bounded on
${\mathcal M}^p_q(\mathbb{R}^n)$, for $1< p\le q<\infty$,
with bound $C(u)\lesssim (1+|u|)^{n/2}$.
We define $\widetilde{\mathcal M}^p_q(\mathbb{R}^n)$
to be the closure of $C^\infty_{\text c}(\mathbb{R}^n)$ in
${\mathcal M}^p_q(\mathbb{R}^n)$,
or equivalently,
$\widetilde{\mathcal M}^p_q(\mathbb{R}^n)$
is the closure of $L^q(\mathbb{R}^n)$ in
${\mathcal M}^p_q(\mathbb{R}^n)$ (see \cite[p.~1846]{YSY15}).
We know that $(-\Delta)^{iu/2}$ maps $L^q(\mathbb{R}^n)$
boundedly into $L^q(\mathbb{R}^n)$ \cite{Gun02}.
We also establish in Lemma \ref{lem:17918-1} that,
for $f \in C^\infty_{\rm c}(\mathbb{R}^n)$,
$\|(-\Delta)^{i u/2}f\|_{{\mathcal M}^p_q}
\lesssim C(u)\|f\|_{{\mathcal M}^p_q}$,
keeping in mind that $C_{\text c}^\infty(\mathbb{R}^n)
\subset L^q(\mathbb{R}^n) \subset
{\mathcal M}^p_q(\mathbb{R}^n)$
and that $(-\Delta)^{i u/2}f$ makes sense for
$f \in C^\infty_{\rm c}(\mathbb{R}^n)$ by \eqref{eq:17917-2}.
This means that
$(-\Delta)^{i u/2}:\widetilde{\mathcal M}^p_q(\mathbb{R}^n) \to
\widetilde{\mathcal M}^p_q(\mathbb{R}^n)$ is bounded
(see Definition \ref{def:17917-0} and Lemma \ref{lem:170921-0}).
Next, we know that the  space ${\mathcal H}^{p'}_{q'}(\mathbb{R}^n)$
is the dual of $\widetilde{\mathcal M}^p_q(\mathbb{R}^n)$
(see \cite[Theorem 4.3]{ST15}) if $\frac1p+\frac{1}{p'}=
\frac1q+\frac{1}{q'}=1$.
Here, we recall that $\mathcal{H}^{p'}_{q'}(\mathbb{R}^n)$ is defined
to be the set of all functions $f\in L^{q'}(\mathbb{R}^n)$ for which
\begin{align}\label{eq:171130-1}
f=\sum_{j=1}^\infty \lambda_j A_j,
\end{align}
where
$\{\lambda_j\}_{j=1}^\infty \in \ell^1$ and
$\{A_j\}_{j=1}^\infty$ is a sequence of functions supported on balls
with $\|A_j\|_{L^{q'}}\le 1$ for every $j\in \mathbb{N}$.
The norm of $f\in {\mathcal H}^{p'}_{q'}$ is defined by
\[
\| f \|_{{\mathcal H}^{p'}_{q'}}
:= \inf \left\{ \sum_{j=1}^\infty |\lambda_j|:  \{\lambda_j\}_{j=1}^\infty
\  {\rm and} \  \{A_j\}_{j=1}^\infty \  {\rm satisfying\  \eqref{eq:171130-1} } \right\}.
\]
Meanwhile, the dual of ${\mathcal H}^{p'}_{q'}(\mathbb{R}^n)$
is ${\mathcal M}^p_q(\mathbb{R}^n)$ \cite{Zo86}.
In general, the dual mapping of a bounded linear mapping
$T$ from a Banach space $X$ to $Y$ is bounded from $Y^*$ to $X^*$.
So, since $(-\Delta)^{i u/2}$ is formally self-adjoint,
we see that the boundedness
$(-\Delta)^{i u/2}:\widetilde{\mathcal M}^p_q(\mathbb{R}^n) \to
\widetilde{\mathcal M}^p_q(\mathbb{R}^n)$
established above entails
$(-\Delta)^{i u/2}:{\mathcal H}^{p'}_{q'}(\mathbb{R}^n) \to
{\mathcal H}^{p'}_{q'}(\mathbb{R}^n)$
(see Definition \ref{def:17917-1} and Lemma \ref{lem:17921-1}),
which in turn entails the boundedness of
$(-\Delta)^{i u/2}:{\mathcal M}^p_q(\mathbb{R}^n) \to
{\mathcal M}^p_q(\mathbb{R}^n)$
(see Definition \ref{def:17921-2} and Proposition \ref{prop:170917}).

We note that $|\cdot|^{i u}\widehat{f}$ does not make sense
for some $f \in {\mathcal M}^p_q(\mathbb{R}^n)$. As indicated
above, the operator $(-\Delta)^{iu/2}$ which is initially defined
on $C_{\text c}^\infty(\mathbb{R}^n)$ is then defined on
$\mathcal{M}^p_q(\mathbb{R}^n)$ by the duality relation
\[
\langle (-\Delta)^{i u/2}f,g \rangle
=
\langle f,(-\Delta)^{-i u/2}g \rangle,
\quad g \in {\mathcal H}^{p'}_{q'}(\mathbb{R}^n),
\]
where ${\mathcal H}^{p'}_{q'}(\mathbb{R}^n)$ is the space
whose dual is ${\mathcal M}^p_q(\mathbb{R}^n)$
(see \cite[Proposition 5]{Zo86} and Definition 2.4).
We claim that this definition of $(-\Delta)^{i u/2}f$ coincides
with the one by the Fourier transform, whenever the Fourier
transform of $f$ makes sense. Indeed, we can show that
\[
\overline{\psi(\xi)}{\mathcal F}[(-\Delta)^{i u/2}f](\xi)=
\overline{\psi(\xi)}|\xi|^{i u}{\mathcal F}f(\xi),
\]
for every $\psi \in C^\infty_{\text c}({\mathbb R}^n)$ and
$0\notin {\rm supp}~\psi$,
where ${\mathcal F}$ denotes the Fourier transform.
Observe that if $g \in {\mathcal H}^{p'}_{q'}(\mathbb{R}^n)$,
then ${\mathcal F}^{-1}[\psi{\mathcal F}g] \in
{\mathcal H}^{p'}_{q'}(\mathbb{R}^n)$.
In fact
$$
{\mathcal F}^{-1}[\psi{\mathcal F}g](x)
=(2\pi)^n
{\mathcal F}^{-1}\psi*g(x)
=(2\pi)^n
\int_{{\mathbb R}^n}{\mathcal F}^{-1}\psi(y)g(x-y)\,dy.
$$
As a result,
\begin{align*}
\|{\mathcal F}^{-1}[\psi{\mathcal F}g]\|_{{\mathcal H}^{p'}_{q'}}
&\le
(2\pi)^n
\int_{{\mathbb R}^n}|{\mathcal F}^{-1}\psi(y)|\|g(\cdot-y)\|_{{\mathcal H}^{p'}_{q'}}\,dy\\
&\le
(2\pi)^n
\int_{{\mathbb R}^n}|{\mathcal F}^{-1}\psi(y)|\|g\|_{{\mathcal H}^{p'}_{q'}}\,dy\\
&=C\|g\|_{{\mathcal H}^{p'}_{q'}}<\infty.
\end{align*}
So,
${\mathcal F}^{-1}[\psi{\mathcal F}g]\in {\mathcal H}^{p'}_{q'}(\mathbb{R}^n)$.
It follows that
\[
\langle (-\Delta)^{i u/2}f,{\mathcal F}^{-1}[\psi{\mathcal F}g] \rangle
=
\langle f,(-\Delta)^{-i u/2}{\mathcal F}^{-1}[\psi{\mathcal F}g] \rangle,
\]
or equivalently
\[
\langle {\mathcal F}^{-1}[\overline{\psi}{\mathcal F}[(-\Delta)^{i u/2}f]],g \rangle
=
\langle f,(-\Delta)^{-i u/2}{\mathcal F}^{-1}[\psi{\mathcal F}g] \rangle.
\]
Since $g \in L^{q'}(\mathbb{R}^n)$, we have
\[
(-\Delta)^{-i u/2}{\mathcal F}^{-1}[\psi{\mathcal F}g]=
{\mathcal F}^{-1}[|\cdot|^{-i u}\psi{\mathcal F}g],
\]
and hence
\[
\langle f,(-\Delta)^{-i u/2}{\mathcal F}^{-1}[\psi{\mathcal F}g] \rangle
=
\langle f,{\mathcal F}^{-1}[|\cdot|^{-i u}\psi{\mathcal F}g]\rangle
=
\langle \mathcal{F}^{-1}[\overline{\psi}|\cdot|^{iu}\mathcal{F}f],g\rangle.
\]
We therefore have
\[
\langle {\mathcal F}^{-1}[\overline{\psi}{\mathcal F}[(-\Delta)^{i u/2}f]],g \rangle
=
\langle {\mathcal F}^{-1}[\overline{\psi}|\cdot|^{i u}{\mathcal F}f],g \rangle.
\]
Since $g$ is arbitrary,
${\mathcal F}^{-1}[\overline{\psi}{\mathcal F}[(-\Delta)^{i u/2}f]]
={\mathcal F}^{-1}[\overline{\psi}|\cdot|^{i u}{\mathcal F}f]$,
so that we obtain $\overline{\psi}{\mathcal F}[(-\Delta)^{i u/2}f]$
$=\overline{\psi}|\cdot|^{i u}{\mathcal F}f$ as claimed.

In the following sections, we prove the Morrey norm estimate for
the imaginary power of the Laplacian and its consequence for
the fractional power of the Laplacian. We also prove a Hardy-type
inequality and use it together with the estimate for the fractional
power of the Laplacian to obtain Heisenberg's uncertainty
inequality in Morrey spaces.

\section{Morrey norm estimates for the fractional power of the
Laplacian}

For each $u\in\mathbb{R}\setminus\{0\}$, it is known that on
$L^p(\mathbb{R}^n)$, for $1\le p\le 2$, the operator
$(-\Delta)^{i u/2}$ (defined by (\ref{eq:17917-2}))
admits an integral kernel $K_{u}$ given by
\[
K_{u}(x):=\frac{\pi^{-n/2}\Gamma\left(\frac{n+iu}{2}\right)}
{2^{-iu}\Gamma\left(\frac{-iu}{2}\right)}|x|^{-n-iu}
=C(u)|x|^{-n-iu},
\quad x\in \mathbb{R}^n
\]
(see \cite[p.~51]{Stein70}).
Here $\widehat{K_u}(\xi)=|\xi|^{iu}$ (in the distribution sense).
A close inspection of the above constant shows
\[
|C(u)|\lesssim (1+|u|)^{\frac{n}{2}}, \quad u\in\mathbb{R}.
\]
As shown in \cite{Gun02, SW01}, we have
\[
\|(-\Delta)^{i u/2}f\|_{L^p}
\lesssim
(1+|u|)^{\left|\frac{n}{p}-\frac{n}{2}\right|}\|f\|_{L^p}
\lesssim
(1+|u|)^{\frac{n}{2}}\|f\|_{L^p},
\quad f\in L^p(\mathbb{R}^n),
\]
for every $u\in\mathbb{R}$, provided that $1<p\le 2$. By duality,
the same inequality also holds for $2<p<\infty$.

Based on the discussion in Section 1,
we shall now prove that the inequality also holds in Morrey spaces
(see \cite{RT14} for similar results).
We need several lemmas and definitions.

\begin{lemma}\label{lem:17918-1}
Let $u\in \mathbb{R}$ and $1<p\le q <\infty$. Then we have
\[
\|(-\Delta)^{iu/2}f\|_{\widetilde{\mathcal{M}}^p_q}
\lesssim
(1+|u|)^{\frac{n}{2}}
\|f\|_{\widetilde{\mathcal{M}}^p_q}
\]
for every $f\in C^\infty_{\rm c} (\mathbb{R}^n)$.
\end{lemma}

\begin{proof}
To prove the inequality, it is sufficient for us to establish
\[
|B(a,r)|^{\frac1q-\frac1p}
\left(\int_{B(a,r)}|(-\Delta)^{i u/2}f(x)|^p\,{\rm d}x\right)^{\frac1p}
\lesssim
(1+|u|)^{\frac{n}{2}}
\|f\|_{{\mathcal M}^p_q}
\]
for all fixed balls $B=B(a,r)$.
To do so, we adopt the technique used in \cite{Nakai94}.
For a fixed ball $B=B(a,r)$, we decompose $f:=f_1+f_2$,
where $f_1:=f\chi_{B(a,2r)}$ and $f_2:=f-f_1$.
Then by the boundedness of $(-\Delta)^{iu/2}$
on $L^p(\mathbb{R}^n)$, we have
\begin{align*}
|B(a,r)|^{\frac1q-\frac1p}
&\left(\int_{B(a,r)}|(-\Delta)^{i u/2}f_1(x)|^p\,{\rm d}x\right)^{\frac1p}
\\
&\le
|B(a,r)|^{\frac1q-\frac1p}
\left(\int_{{\mathbb R}^n}|(-\Delta)^{i u/2}f_1(x)|^p\,{\rm d}x\right)^{\frac1p}\\
&\lesssim
(1+|u|)^{\frac{n}{2}}|B(a,r)|^{\frac1q-\frac1p}
\left(\int_{{\mathbb R}^n}|f_1(x)|^p\,{\rm d}x\right)^{\frac1p}\\
&\sim
(1+|u|)^{\frac{n}{2}}|B(a,2r)|^{\frac1q-\frac1p}
\left(\int_{B(a,2r)}|f(x)|^p\,{\rm d}x\right)^{\frac1p}\\
&\lesssim
(1+|u|)^{\frac{n}{2}}\|f\|_{{\mathcal M}^p_q}.
\end{align*}
Meanwhile, for each $x \in B$, we have
\begin{align*}
|(-\Delta)^{i u/2}f_2(x)|
&
\le |C(u)|
\int_{{\mathbb R}^n \setminus B(x,r)}
\frac{|f(y)|}{|x-y|^n}\,{\rm d}y
\\
&\le |C(u)|
\sum_{k=0}^\infty
\int_{B(x,2^{k+1}r) \setminus B(x,2^k r)}
\frac{|f(y)|}{|x-y|^n}\,{\rm d}y\\
&\lesssim |C(u)|
\sum_{k=0}^\infty
\frac{1}{(2^k r)^n}
\int_{B(x,2^{k+1}r) \setminus B(x,2^k r)}|f(y)|\,{\rm d}y\\
&\lesssim |C(u)|
\sum_{k=0}^\infty
\left(
\frac{1}{(2^k r)^n}
\int_{B(x,2^{k+1}r) \setminus B(x,2^k r)}|f(y)|^p\,{\rm d}y
\right)^{\frac1p}\\
&\lesssim |C(u)|\|f\|_{{\mathcal M}^p_q}
\sum_{k=0}^\infty (2^k r)^{-\frac{n}{q}}\\
&\lesssim r^{-\frac{n}{q}}|C(u)|\|f\|_{{\mathcal M}^p_q}.
\end{align*}
Thus,
\begin{align*}
|B(a,r)|^{\frac1q-\frac1p}
&\left(\int_{B(a,r)}|(-\Delta)^{i u/2}f_2(x)|^p\,{\rm d}x\right)^{\frac1p}
\\
&\lesssim
|B(a,r)|^{\frac1q-\frac1p}
\left(\int_{B(a,r)}(r^{-\frac{n}{q}}|C(u)|\|f\|_{{\mathcal M}^p_q})^p
\,{\rm d}y\right)^{\frac1p}\\
&=
|B(a,r)|^{\frac1q}r^{-\frac{n}{q}}|C(u)|\|f\|_{{\mathcal M}^p_q}\\
&\sim
|C(u)|\|f\|_{{\mathcal M}^p_q}\\
&\lesssim
(1+|u|)^{\frac{n}{2}}\|f\|_{{\mathcal M}^p_q}.
\end{align*}
Combining the two estimates, we obtain the desired inequality.
\end{proof}

Using Lemma \ref{lem:17918-1} and density,
we give the following natural definition:

\begin{definition}\label{def:17917-0}
Given $f\in \widetilde{\mathcal{M}}^p_q(\mathbb{R}^n)$, we define
\[
(-\Delta)^{iu/2} f:= \lim_{j\to \infty} (-\Delta)^{iu/2} f_j,
\]
where $f_j\in C_{\text c}^\infty(\mathbb{R}^n)$ and $f_j\to f$ in
the $\mathcal{M}^p_q$-norm.
\end{definition}

A direct consequence of Lemma \ref{lem:17918-1}
and the above definition is:

\begin{lemma}\label{lem:170921-0}
Let $u\in \mathbb{R}$ and $1<p\le q <\infty$. Then we have
\[
\|(-\Delta)^{iu/2}f\|_{\widetilde{\mathcal{M}}^p_q}
\lesssim
(1+|u|)^{\frac{n}{2}}
\|f\|_{\widetilde{\mathcal{M}}^p_q}
\]
for every $f\in  \widetilde{\mathcal{M}}^p_q(\mathbb{R}^n)$.
\end{lemma}

\begin{definition}\label{def:17917-1}
For every $g\in \mathcal{H}^{p'}_{q'}(\mathbb{R}^n)$, we define
\[
\langle (-\Delta)^{iu/2} g, h\rangle
=
\langle g,  (-\Delta)^{-iu/2} h\rangle,
\]
for every $h\in \widetilde{\mathcal{M}}^p_q(\mathbb{R}^n)$.
\end{definition}

\begin{lemma}\label{lem:17921-1}
Let $u\in \mathbb{R}$ and $1<p\le q <\infty$. Then
\[
\|(-\Delta)^{iu/2} g\|_{\mathcal{H}^{p'}_{q'}}
\lesssim
(1+|u|)^{\frac{n}{2}}
\|g\|_{\mathcal{H}^{p'}_{q'}}
\]
for every $g\in \mathcal{H}^{p'}_{q'}(\mathbb{R}^n)$.
\end{lemma}

\begin{proof}
For every $h\in \widetilde{\mathcal{M}}^p_q(\mathbb{R}^n)$, we have
\begin{align*}
|\langle (-\Delta)^{iu/2} g, h\rangle|
&=
|\langle g,  (-\Delta)^{-iu/2} h\rangle|
\le
\|g\|_{\mathcal{H}^{p'}_{q'}}
\| (-\Delta)^{-iu/2} h\|_{\widetilde{\mathcal{M}}^p_q}\\
&\lesssim
(1+|u|)^{\frac{n}{2}}
\|g\|_{\mathcal{H}^{p'}_{q'}}
\| h\|_{\widetilde{\mathcal{M}}^p_q}.
\end{align*}
Since $(\widetilde{\mathcal{M}^p_q})^*(\mathbb{R}^n) \simeq
\mathcal{H}^{p'}_{q'}(\mathbb{R}^n)$
\cite{Zo86}, we get the desired result.
\end{proof}

We use Lemma \ref{lem:17921-1} to give the following definition:

\begin{definition}\label{def:17921-2}
For every $f\in \mathcal{M}^{p}_{q}(\mathbb{R}^n)$, we define
\[
\langle (-\Delta)^{iu/2} f, g\rangle
=
\langle f,  (-\Delta)^{-iu/2} g\rangle,
\]
for every $g\in \mathcal{H}^{p'}_{q'}(\mathbb{R}^n)$.
\end{definition}

\begin{proposition}\label{prop:170917}
Let $u\in \mathbb{R}$ and $1<p\le q <\infty$. Then
\[
\|(-\Delta)^{iu/2} f\|_{\mathcal{M}^{p}_{q}}
\lesssim
(1+|u|)^{\frac{n}{2}}
\|f\|_{\mathcal{M}^{p}_{q}}
\]
for every $f\in \mathcal{M}^{p}_{q}(\mathbb{R}^n)$.
\end{proposition}

\begin{proof}
For every $g\in \mathcal{H}^{p'}_{q'}(\mathbb{R}^n)$, we have
\begin{align*}
|\langle (-\Delta)^{iu/2} f, g\rangle|
&=
|\langle f,  (-\Delta)^{-iu/2} g\rangle|
\le
\|f\|_{\mathcal{M}^{p}_{q}}
\| (-\Delta)^{-iu/2} g\|_{\mathcal{H}^{p'}_{q'}}\\
&\lesssim
(1+|u|)^{\frac{n}{2}}
\| f\|_{\mathcal{M}^p_q}
\|g\|_{\mathcal{H}^{p'}_{q'}}.
\end{align*}
Since $(\mathcal{H}^{p'}_{q'})^*(\mathbb{R}^n) \simeq
\mathcal{M}^{p}_{q}(\mathbb{R}^n)$, we get the desired result.
\end{proof}

As a corollary of Proposition \ref{prop:170917}, we obtain the following
result for the fractional power of the Laplacian, which is analogous to
the interpolation inequality in \cite{Ciatti15}.
We refer the interested reader to \cite{LYY14} and references
therein for the interpolation of Morrey spaces.
\begin{theorem}\label{thm:1798-1}
Let $\alpha \ge 0$. Then, for $0\le \theta\le 1$, we have
\begin{equation}\label{est-2}
\|(-\Delta)^{\alpha\theta/2}f\|_{{\mathcal M}^{p}_{q}}
\lesssim
\|f\|^{1-\theta}_{{\mathcal M}^{p_0}_{q_0}}
\|(-\Delta)^{\alpha/2}f\|^\theta_{{\mathcal M}^{p_1}_{q_1}},
\quad f\in C_{\text c}^\infty(\mathbb{R}^n),
\end{equation}
where
\begin{align}\label{eq:1795-1}
\frac{1}{p}=\frac{1-\theta}{p_0}+\frac{\theta}{p_1}, \quad
\frac{1}{q}=\frac{1-\theta}{q_0}+\frac{\theta}{q_1}
\end{align}
with $1 < p_0 \le q_0<\infty$ and $1 < p_1 \le q_1<\infty$.
\end{theorem}

To prove Theorem \ref{thm:1798-1}, we use the following observation
which is based on \cite{MRZ16}.

\begin{lemma}\label{lem:1798-1}
Let $1\le w\le \infty$, $v\in [0,1]$, $\alpha\ge 0$, and
$B$ be any ball in $\mathbb{R}^n$.
Then for every $f\in C^\infty_{\rm c}(\mathbb{R}^n)$,
we have
\[
\|(-\Delta)^{\frac{\alpha v}{2}} f\|_{L^w(B)} \le C,
\]
where the constant $C=C(n,\alpha,B,f)$ is independent of $w$ and $v$.
\end{lemma}

\begin{proof}
Let $N:=\lfloor n+\alpha\rfloor+1$. Then, for every $x\in \mathbb{R}^n$
we have
\begin{align}\label{eq:1798-3}
|(-\Delta)^{\frac{\alpha v}{2}}f(x)|
&\le
\int_{\{|\xi|<1 \}} |\xi|^{\alpha v} |\hat{f}(\xi)| \,{\rm d}\xi
+
\int_{\{|\xi|\ge 1 \}} |\xi|^{\alpha v} |\hat{f}(\xi)| \,{\rm d}\xi
\nonumber
\\
&\le
\|\hat{f}\|_{L^\infty} |B(0,1)| +
\|\mathcal{F}[(-\Delta)^Nf]\|_{L^\infty}
\int_{\{|\xi|\ge 1 \}} |\xi|^{\alpha -2N} \,{\rm d}\xi.
\end{align}
Let $E:={\rm supp}(f)$. Observe that
\begin{align}\label{eq:1798-2}
\|\hat{f}\|_{L^\infty}\le \|f\|_{L^1} \le \|f\|_{C^\infty_{\rm c}(\mathbb{R}^n)}|E|
\end{align}
and
\begin{align}\label{eq:1798-4}
\|\mathcal{F}[(-\Delta)^Nf]\|_{L^\infty}\le \|(-\Delta)^Nf\|_{L^1} \le
\|f\|_{C^\infty_{\rm c}(\mathbb{R}^n)}|E|.
\end{align}
Combining \eqref{eq:1798-3}-\eqref{eq:1798-4} and
$\int_{\{|\xi|\ge 1 \}} |\xi|^{\alpha -2N} \,{\rm d}\xi=
{\rm O}\left(\frac{1}{2N-\alpha-n}\right)$,
we get
\[
\|(-\Delta)^{\frac{\alpha v}{2}}f\|_{L^\infty(B)}
\le
C_{n,\alpha, f},
\]
where
\[
C_{n,\alpha, f}:=
\left(|B(0,1)|+ \frac{D}{2N-\alpha -n}\right) \|f\|_{C^\infty_{\rm c}(\mathbb{R}^n)}|E|
\]
with $D \gg 1$.
Consequently, for $1\le w<\infty$, we have
\[
\|(-\Delta)^{\frac{\alpha v}{2}}f\|_{L^w(B)}
\le C_{n,\alpha, f} |B|^{\frac1w}
\le C_{n,\alpha, f} \max(1, |B|),
\]
as desired.
\end{proof}

Now we are ready to prove Theorem \ref{thm:1798-1}.

\begin{proof}[Proof of Theorem \ref{thm:1798-1}]
Let $f\in C_{\text c}^\infty(\mathbb{R}^n)$.
We prove \eqref{est-2} by showing that
\begin{align}\label{eq:1795-2}
\left( \int_B |(-\Delta)^{\alpha \theta/2}f (x)|^p \,{\rm d}x\right)^{\frac1p}
\lesssim
|B|^{\frac1p-\frac1q}
\|f\|_{\mathcal{M}^{p_0}_{q_0}}^{1-\theta}
\|(-\Delta)^{\alpha/2}f\|_{\mathcal{M}^{p_1}_{q_1}}^{\theta},
\end{align}
for every fixed ball $B=B(a,r)$.
Let $p_0'$, $p_1'$, and $p'$ be defined by
$\frac{1}{p_0'}:=1-\frac{1}{p_0}$, $\frac{1}{p_1'}:=1-\frac{1}{p_1}$,
and $\frac{1}{p'}:=1-\frac{1}{p}$, respectively.
We define $S:=\{ z\in \mathbb{C}: 0 < {\rm Re}(z) < 1\}$
and let $\overline{S}$ be its closure.
For every $z\in \overline{S}$ and $x\in \mathbb{R}^n$, we define
\[
G(z,x)
:=
\begin{cases}
0, &\quad g(x)=0,
\\
{\rm sgn}(g(x)) |g(x)|^{p'\left(\frac{1-z}{p_0'}+
\frac{z}{p_1'}\right)}, &\quad g(x)\neq 0,
\end{cases}
\]
where $g$ is an arbitrary simple function with
$\|g\|_{L^{p'}(B)}=1$.
We shall apply the Three Lines Theorem to the function
$F(z)$, defined by
\[
F(z):=e^{z^2} \int_B
(-\Delta)^{\alpha z/2} f(x) G(z,x) \,{\rm d}x.
\]
Note that $F$ is continuous on $\overline{S}$ and holomorphic in $S$.
Let $z=v+iu$ where $v\in [0,1]$ and $u\in \mathbb{R}$.
Define
$\frac{1}{w}:=1-\frac{1-v}{p_0'}-\frac{v}{p_1'}$.
Then
\begin{align}\label{eq:1797-2}
|F(v+iu)|
\lesssim e^{-u^2} (1+\alpha |u|)^{\frac{n}{2}}
\| (-\Delta)^{\alpha v/2}f\|_{L^w(B)} \|G(v+iu, \cdot)\|_{L^{w'}(B)}.
\end{align}
Here we have used the boundedness of $(-\Delta)^{i \alpha u/2}$
on $L^w(B)$ and the fact that $$(-\Delta)^{\alpha z/2}=
(-\Delta)^{i \alpha u/2}
(-\Delta)^{\alpha v/2}.$$
Combining \eqref{eq:1797-2}, Lemma \ref{lem:1798-1}, and
\[
\|G(v+iu, \cdot)\|_{L^{w'}(B)}
=
\left\| |g|^{p' \left( \frac{1-v}{p_0'} +
\frac{v}{p_1'} \right)} \right\|_{L^{w'}(B)}
=
\|g\|_{L^{p'}(B)}^{\frac{p'}{w'}}=1,
\]
we have
$
\sup_{z\in \overline{S}} |F(z)|<\infty,
$
that is, $F$ is bounded on $\overline{S}$.
Next, we observe that
\begin{align*}
|F(iu)|
&\lesssim
e^{-u^2} \|(-\Delta)^{i\alpha u/2}f\|_{\mathcal{M}^{p_0}_{q_0}}
|B|^{\frac{1}{p_0}-\frac{1}{q_0}}
\|G(iu, \cdot)\|_{L^{p_0'}(B)}
\\
&\lesssim
e^{-u^2}
(1+\alpha |u|)^{\frac{n}{2}}
\|f\|_{\mathcal{M}^{p_0}_{q_0}}|B|^{\frac{1}{p_0}-\frac{1}{q_0}}
\||g|^{p'/p_0'}\|_{L^{p_0'}(B)}
\\
&\lesssim
\|f\|_{\mathcal{M}^{p_0}_{q_0}}|B|^{\frac{1}{p_0}-\frac{1}{q_0}}
\end{align*}
and similarly
\[
|F(1+iu)|\lesssim
\|(-\Delta)^{\alpha/2}f\|_{\mathcal{M}^{p_1}_{q_1}}
|B|^{\frac{1}{p_1}-\frac{1}{q_1}}.
\]
It thus follows from the Three Lines Theorem that
\begin{align*}
|F(\theta)|
&\le \sup_{u\in \mathbb{R}}|F( \theta+iu)|\\
&\le \left(\sup_{u\in \mathbb{R}}|F(iu)|\right)^{1-\theta}
\cdot
\left(\sup_{u\in \mathbb{R}}|F(1+iu)|\right)^{\theta}\\
&\lesssim
\|f\|_{\mathcal{M}^{p_0}_{q_0}}^{1-\theta}
\|(-\Delta)^{\alpha/2}f\|_{\mathcal{M}^{p_1}_{q_1}}^{\theta}
|B|^{\frac1p-\frac1q},
\end{align*}
for $0\le \theta\le 1$.
Accordingly, we obtain
\[
\left|
\int_{B} (-\Delta)^{\alpha \theta/2} f(x) g(x) \,{\rm d}x
\right|
= e^{-\theta^2} |F(\theta)|
\lesssim
\|f\|_{\mathcal{M}^{p_0}_{q_0}}^{1-\theta}
\|(-\Delta)^{\alpha/2}f\|_{\mathcal{M}^{p_1}_{q_1}}^{\theta}
|B|^{\frac1p-\frac1q}.
\]
Since $g$ is any simple function of $L^{p'}(B)$-norm 1, we conclude
that \eqref{eq:1795-2} holds.
\end{proof}

\section{A Hardy-type inequality and a Heisenberg-type inequality}

We shall now prove a Hardy-type inequality and Heisenberg's
uncertainty inequality in Morrey spaces.
According to \cite{SST11}, we have
\begin{equation}\label{ineq:170918}
\|W\cdot (-\Delta)^{-\alpha/2}f\|_{\mathcal{M}^p_q}
\lesssim
\|W\|_{\mathcal{M}^u_v} \|f\|_{\mathcal{M}^p_q},
\quad f\in \mathcal{M}^p_q(\mathbb{R}^n),
\end{equation}
where $0<\alpha<n,\ 1<p\le q<\frac{n}{\alpha},\
u=\frac{n p}{\alpha q},\
v=\frac{n}{\alpha}$.
This inequality goes back to the work of Olsen \cite{Olsen95},
so we call it Olsen's inequality.
Note that the inequality follows from H\"older's inequality and the
boundedness of the fractional integral operator $I_\alpha:=
(-\Delta)^{-\alpha/2}$ from $\mathcal{M}^p_q(\mathbb{R}^n)$
to $\mathcal{M}^s_t(\mathbb{R}^n)$
for $0<\alpha<n,\ 1<p\le q<\frac{n}{\alpha}$,
$\frac{1}{s}=\frac{1}{p}-\frac{\alpha q}{n p}$, and
$\frac{s}{t}=\frac{p}{q}$ (see also \cite{GE09}).
Note that through its Fourier transform, one may recognize
$(-\Delta)^{-\alpha/2}$ as the convolution operator whose
kernel is a multiple of $|\cdot|^{\alpha-n}$, which is initially
defined on $C_{\text c}^\infty(\mathbb{R}^n)$ (see \cite{Stein70}).

As a consequence of the inequality (\ref{ineq:170918}), we have:

\begin{proposition}\label{prop-2}
Let $1<p\le q<\infty$ and $0<\alpha<\frac{n}{q}$.
Then we have
\begin{equation}\label{est-3}
\||\cdot|^{-\alpha} g\|_{\mathcal{M}^p_q}
\lesssim
\|(-\Delta)^{\alpha/2}g\|_{\mathcal{M}^p_q}
\end{equation}
for every $g \in C_{\text c}^\infty(\mathbb{R}^n)$.
\end{proposition}

\begin{remark}
The inequality (\ref{est-3}) may be viewed as
a Hardy-type inequality in Morrey spaces.
\end{remark}

To prove the proposition, we need some lemmas.

\begin{lemma}
Let $0<\alpha<n$. If $g \in C^\infty_{\rm c}({\mathbb R}^n)$,
then we have
\[
|(-\Delta)^{\alpha/2}g(x)| \lesssim \min(1,|x|^{-\alpha-n}).
\]
In particular, $f=(-\Delta)^{\alpha/2}g \in L^1({\mathbb R}^n) \cap
L^\infty({\mathbb R}^n)$.
\end{lemma}

\begin{proof}
We have already seen that $|(-\Delta)^{\alpha/2}g(x)| \lesssim 1$
in the proof of  Lemma  \ref{lem:1798-1}.
Now let $\psi \in C^\infty_{\text c}({\mathbb R}^n)$
be such that
$\chi_{B(1)} \le \psi \le \chi_{B(2)}$, where $B(r)$
denotes the ball centered at the origin of radius $r$.
Define
$\varphi_j(\xi)=\psi(2^{-j}\xi)-\psi(2^{-j+1}\xi)$.
We decompose
\[
(-\Delta)^{\alpha/2}g(x)
=
{\mathcal F}^{-1}[|\cdot|^{\alpha}(1-\psi){\mathcal F}g](x)
+
\sum_{j=-\infty}^0
{\mathcal F}^{-1}[|\cdot|^{\alpha}\varphi_j{\mathcal F}g](x).
\]
Since
$h={\mathcal F}^{-1}[|\cdot|^{\alpha}(1-\psi){\mathcal F}g]$
belongs to ${\mathcal S}({\mathbb R}^n)$,
we need to handle the second term.
Using a crude estimate
${\mathcal F}g \in L^\infty({\mathbb R}^n)$,
we get
\[
|{\mathcal F}^{-1}[|\cdot|^{\alpha}\varphi_j{\mathcal F}g](x)|
\lesssim
2^{j\alpha}
\||2^{-j}\cdot|^{\alpha}\varphi_j{\mathcal F}g\|_{L^1}
\sim 2^{j(\alpha+n)}.
\]
Let $N \in {\mathbb N}$ be large enough.
Then as before,
\begin{align*}
|x|^{2N}|{\mathcal F}^{-1}[|\cdot|^{\alpha}\varphi_j{\mathcal F}g](x)|
&=
|{\mathcal F}^{-1}[\Delta^N[|\cdot|^{\alpha}\varphi_j{\mathcal F}g]](x)|\\
&\lesssim
\sum_{\beta \in ({\mathbb N} \cup \{0\})^n, |\beta|=2N}
\|\partial^\beta[|\cdot|^{\alpha}\varphi_j{\mathcal F}g]\|_{L^1}.
\end{align*}
Here and below let $\beta$ be such that $|\beta|=2N$.
Then
\[
|\partial^\beta[|\xi|^{\alpha}\varphi_j(\xi){\mathcal F}g(\xi)]|
\lesssim
\sum_{\beta_1+\beta_2+\beta_3=\beta}
|\partial^{\beta_1}[|\xi|^{\alpha}]|
|\partial^{\beta_2}\varphi_j(\xi)|
|\partial^{\beta_3}{\mathcal F}g(\xi)|.
\]
Noting that
$\varphi_j(\xi)$ vanishes outside
$\{2^{j-2} \le |\xi| \le 2^{j+2}\}$,
we have
\[
\partial^{\beta_1}[|\xi|^{\alpha}]={\rm O}(|\xi|^{\alpha-|\beta_1|}), \quad
\partial^{\beta_2}\varphi_j(\xi)={\rm O}(|\xi|^{-|\beta_2|}), \quad
|\partial^{\beta_3}{\mathcal F}g(\xi)|\lesssim 1 \lesssim 2^{-j|\beta_3|},
\]
as $\xi \to 0$.
Thus,
\begin{align*}
|\partial^\beta[|\xi|^{\alpha}\varphi_j(\xi){\mathcal F}g(\xi)]|
&\ \lesssim
\sum_{\beta_1+\beta_2+\beta_3=\beta}
|\xi|^{\alpha-|\beta_1|}
|\xi|^{-|\beta_2|}
2^{-j|\beta_3|}\chi_{\{2^{j-2} \le |\xi|\le 2^{j+2}\}}(\xi)\\
&\lesssim
2^{j(\alpha-2N)}\chi_{\{|\xi| \le 2^{j+2}\}}(\xi)
\end{align*}
and hence
\[
\|\partial^\beta[|\cdot|^{\alpha}\varphi_j{\mathcal F}g]\|_{L^1}=
{\rm O}(2^{j(\alpha+n-2N)})
\]
as $j \to -\infty$.
As a result,
\begin{align*}
|(-\Delta)^{\alpha/2}g(x)|
&\lesssim
|x|^{-\alpha-n}+
\sum_{j=-\infty}^0
\min(
|x|^{-2N}2^{j(\alpha+n-2N)},2^{j(\alpha+n)})\\
&\le
|x|^{-\alpha-n}
+
|x|^{-\alpha-n}
\sum_{j=-\infty}^\infty
\min(
|x|^{\alpha+n-2N}2^{j(\alpha+n-2N)},|x|^{\alpha+n}2^{j(\alpha+n)}).
\end{align*}
Noticing that
\begin{align*}
\sum_{j=-\infty}^\infty
&\min(
|x|^{\alpha+n-2N}2^{j(\alpha+n-2N)},|x|^{\alpha+n}2^{j(\alpha+n)})
\\
&\le
\sum_{j=-\infty; 2^j|x|\le 1}^\infty
(2^j|x|)^{\alpha+n}
+\sum_{j=-\infty; 2^j|x|>1}^\infty
(2^j|x|)^{\alpha+n-N}
\\
&\lesssim
\sum_{j=-\infty; 2^j|x|\le 1}^\infty
\int_{2^j|x|}^{2^{j+1}|x|} t^{\alpha+n-1} \ dt
+\sum_{j=-\infty; 2^j|x|>1}^\infty
\int_{2^{j-1}|x|}^{2^{j}|x|} t^{\alpha+n-N-1} \ dt
\\
&\le
\int_{0}^{2} t^{\alpha+n-1} \ dt
+
\int_{1/2}^{\infty} t^{\alpha+n-N-1} \ dt
\lesssim 1,
\end{align*}
we conclude that
\[
|(-\Delta)^{\alpha/2}g(x)|
\lesssim|x|^{-\alpha-n},
\]
as desired.
\end{proof}

\begin{lemma}\label{lemma:170919}
Let $1\le p\le q<\infty$ and $0<\alpha <n$.
For $g\in C^{\infty}_{\rm c}(\mathbb{R}^n)$,
define $f:=(-\Delta)^{\alpha/2}g$.
Then $f\in \mathcal{M}^p_q(\mathbb{R}^n)$ and
$(-\Delta)^{-\alpha/2}f=g$ pointwise.
\end{lemma}

\begin{proof}
We have proved that
$f \in L^1({\mathbb R}^n) \cap L^\infty({\mathbb R}^n)$.
Consequently,
\[
\|f\|_{\mathcal{M}^p_q} \le \|f\|_{L^q} \le \|f\|_{L^\infty}^{1-\frac1q}
\|f\|_{L^1}^{\frac1q}<\infty.
\]
[This justifies the right-hand side of \eqref{est-3}.]
Next, $|\cdot|^\alpha\widehat{g} \in L^1(\mathbb{R}^n)$ and
$f=\mathcal{F}^{-1}(|\cdot|^\alpha \widehat{g})\in L^1(\mathbb{R}^n)$.
Hence $\widehat{f}=|\cdot|^\alpha\widehat{g}$ pointwise,
and so $|\cdot|^{-\alpha}\widehat{f}=\widehat{g}$ pointwise.
This tells us that $(-\Delta)^{-\alpha/2}f=g$ pointwise.
\end{proof}

Now we come to the proof of Proposition \ref{prop-2}.

\begin{proof}[Proof of Proposition \ref{prop-2}]
For $1<p<q<\infty$ and $0<\alpha<\frac{n}{q}$, we have
$u=\frac{n p}{\alpha q}<\frac{n}{\alpha}=v$.
By computing directly its Morrey norm, we obtain that
$W(\cdot):=|\cdot|^{-\alpha} \in \mathcal{M}^u_v(\mathbb{R}^n)$.
Hence, for $g\in C_{\text c}^\infty(\mathbb{R}^n)$, we
take $f:=(-\Delta)^{\alpha/2}g$,
which is a function in
$\mathcal{M}^p_q(\mathbb{R}^n)$ by Lemma \ref{lemma:170919}.
Moreover, $g=(-\Delta)^{-\alpha/2}f \in \mathcal{M}^s_t(\mathbb{R}^n)$
where $\frac{1}{s}=\frac{1}{p}-\frac{\alpha q}{np}$ and $\frac{s}{t}=
\frac{p}{q}$, so that Olsen's inequality (\ref{ineq:170918}) gives
\[
\||\cdot|^{-\alpha} g\|_{\mathcal{M}^p_q}
\lesssim \|W\|_{\mathcal{M}^u_v}
\|(-\Delta)^{\alpha/2}g\|_{\mathcal{M}^p_q}.
\]
For $1\le p=q<\frac{n}{\alpha}$, we use the fact that
$f\in L^q(\mathbb{R}^n)$ and that $g=(-\Delta)^{-\alpha/2}f
\in wL^t(\mathbb{R}^n)$ for $\frac{1}{t}=\frac{1}{q}-\frac{\alpha}{n}$
with $\|(-\Delta)^{-\alpha/2}f\|_{wL^t} \lesssim \|f\|_{L^q}$ (where
$wL^t(\mathbb{R}^n)$ denotes the weak Lebesgue space of exponent $t$).
It thus follows from \cite[Proposition 4.1]{Nakai17} that
\[
\||\cdot|^{-\alpha}g\|_{wL^q}=\|W(-\Delta)^{-\alpha/2}f\|_{wL^q}
\lesssim \|W\|_{wL^v}
\|(-\Delta)^{-\alpha/2}f\|_{wL^t} \lesssim \|W\|_{wL^v}\|f\|_{L^q},
\]
where $v=\frac{n}{\alpha}$ (as above). This inequality holds for
every $1\le q<\frac{n}{\alpha}$. By the Marcinkiewicz interpolation
theorem, we obtain
\[
\||\cdot|^{-\alpha}g\|_{L^q}\lesssim \|W\|_{wL^v}\|f\|_{L^q} =
\|W\|_{wL^v}\|(-\Delta)^{\alpha/2}g\|_{L^q},
\]
for $1<q<\frac{n}{\alpha}$. This completes the proof.
\end{proof}

As a corollary of Proposition \ref{prop-2}, we obtain the following
result (which is analogous to \cite[Corollary 5.2]{Ciatti15}).

\begin{theorem}\label{thm-3}
Let $1<p\le q<\infty$, $1\le p_2\le q_2<\infty$, $\beta>0$, and
$0<\gamma<\frac{n}{q}$.
If $\frac{\beta+\gamma}{p_0}=\frac{\beta}{p}+\frac{\gamma}{p_2}$ and
$\frac{\beta+\gamma}{q_0}=\frac{\beta}{q}+\frac{\gamma}{q_2}$,
then
\[
\|g\|_{\mathcal{M}^{p_0}_{q_0}}
\lesssim
\||\cdot|^{\beta}g\|^{\gamma/(\beta+\gamma)}_{\mathcal{M}^{p_2}_{q_2}}
\|(-\Delta)^{\gamma/2}g\|^{\beta/(\beta+\gamma)}_{\mathcal{M}^p_q}
\]
for every $g \in C_{\text c}^\infty(\mathbb{R}^n)$.
\end{theorem}

\begin{proof}
Write $g(x)=\bigl[|x|^{\beta}g(x)\bigr]^{\gamma/(\beta+\gamma)}
\bigl[|x|^{-\gamma}g(x)\bigr]^{\beta/(\beta+\gamma)}$.
By H\"older's inequality and Proposition \ref{prop-2}, we obtain
\[
\|g\|_{\mathcal{M}^{p_0}_{q_0}}
\le
\||\cdot|^{\beta}g\|^{\gamma/(\beta+\gamma)}_{\mathcal{M}^{p_2}_{q_2}}
\||\cdot|^{-\gamma}g\|^{\beta/(\beta+\gamma)}_{\mathcal{M}^{p}_{q}}\\
\lesssim
\||\cdot|^{\beta}g\|^{\gamma/(\beta+\gamma)}_{\mathcal{M}^{p_2}_{q_2}}
\|(-\Delta)^{\gamma/2}g\|^{\beta/(\beta+\gamma)}_{\mathcal{M}^{p}_{q}},
\]
as desired.
\end{proof}

Finally, we use our estimate for the fractional power of the Laplacian
in Theorem \ref{thm:1798-1} to get the following Heisenberg's
uncertainty inequality (which is analogous to \cite[Theorem 5.4]{Ciatti15}).

\begin{theorem}\label{thm-4}
Let $1<p_1\le q_1<\infty$, $1\le p_2\le q_2<\infty$, and $\beta, \delta>0$.
If $\frac{\beta+\delta}{p_0}=\frac{\beta}{p_1}+\frac{\delta}{p_2}$ and
$\frac{\beta+\delta}{q_0}=\frac{\beta}{q_1}+\frac{\delta}{q_2}$, then
\[
\|g\|_{\mathcal{M}^{p_0}_{q_0}}
\lesssim
\||\cdot|^{\beta}g\|^{\delta/(\beta+\delta)}_{\mathcal{M}^{p_2}_{q_2}}
\|(-\Delta)^{\delta/2}g\|^{\beta/(\beta+\delta)}_{\mathcal{M}^{p_1}_{q_1}}
\]
for every $g\in C_{\text c}^\infty(\mathbb{R}^n)$.
\end{theorem}

\begin{proof}
The idea of the proof is the same as in \cite{Ciatti15}. If $\delta
<\frac{n}{q_1}$, we do not have to do anything -- the inequality is
the same as in Theorem \ref{thm-3}. Otherwise, we set
$\gamma=\delta\theta$ and apply the interpolation inequality
\[
\|(-\Delta)^{\delta \theta/2}g\|_{{\mathcal M}^p_q} \lesssim
\|g\|_{{\mathcal M}^{p_0}_{q_0}}^{1-\theta}
\|(-\Delta)^{\delta/2} g\|_{{\mathcal M}^{p_1}_{q_1}}^\theta
\]
for $0<\theta<\frac{n}{\delta q_1}$, so that the inequality in
Theorem \ref{thm-3} becomes
\[
\|g\|_{\mathcal{M}^{p_0}_{q_0}}
\lesssim
\||\cdot|^{\beta}g\|^{\gamma/(\beta+\gamma)}_{\mathcal{M}^{p_2}_{q_2}}
\|(-\Delta)^{\delta/2}g\|^{\beta\theta/(\beta+\gamma)}_{\mathcal{M}^{p_1}_{q_1}}
\|g\|_{\mathcal{M}^{p_0}_{q_0}}^{\beta(1-\theta)/(\beta+\gamma)}.
\]
Rearranging the expression, we get the desired inequality.
\end{proof}

\begin{remark}
Note that the value of $\delta$ in the above proposition can be
as large as possible. This is the benefit from the interpolation
inequality for the fractional power of the Laplacian.
\end{remark}

\bigskip

\noindent{\bf Acknowledgement}.
The authors would like to thank the referee for her/his useful comments.

\end{document}